\documentclass{article}

\usepackage{a4wide}

\usepackage{amsfonts,amsmath,amssymb,amsbsy,enumerate,natbib,amsthm}

\def\etal{\mbox{et al.}}

\newcommand{\R}{\mathbb{R}}
\newcommand{\T}{\mathbb{T}}
\newcommand{\rar}{\mbox{$\rightarrow$}}

\newcommand{\Z}{\mathbb{Z}}

\newtheorem{theorem}{Theorem}[section]
\newtheorem{lemma}[theorem]{Lemma}

\newtheorem{proposition}[theorem]{Proposition}
\newtheorem{definition}[theorem]{Definition}
\newtheorem{example}[theorem]{Example}
\newtheorem{remark}{Remark}

\begin{document}

\title{Backward Linear Control Systems on Time Scales\thanks{Submitted November 11, 2009; 
Revised March 28, 2010; Accepted April 03, 2010; for publication 
in the \emph{International Journal of Control}.}}

\author{Ewa Paw{\l}uszewicz$^{a,b}$\\
\texttt{ewa@ua.pt, epaw@pb.edu.pl}
\and
Delfim F. M. Torres$^{a}$\\
\texttt{delfim@ua.pt}}

\date{$^{a}$Department of Mathematics, 
University of Aveiro\\ 
3810-193 Aveiro, Portugal\\[0.3cm]
$^{b}$Faculty of Computer Science,
Bia{\l}ystok University of Technology\\ 
15-351 Bia{\l}ystok, Poland}

\maketitle


\begin{abstract}
We show how a linear control systems theory
for the backward nabla differential operator
on an arbitrary time scale
can be obtained via Caputo's duality.
More precisely, we consider
linear control systems with outputs
defined with respect to the backward jump operator.
Kalman criteria of controllability
and observability, as well as realizability conditions,
are proved.
\end{abstract}

\bigskip


\smallskip


\smallskip


\section{Introduction}

The theory of linear control systems of both continuous- and discrete-time cases
is a subject well developed -- see, \textrm{e.g.},
\citep{Kalman,LMS_wolovich,O,Zabczyk} and references therein.
It can be noticed that many results obtained in both discrete and continuous
cases are similar or even identical. Recently, many problems in control
theory have been generalized to \emph{time scales} \citep{D2,D1,BPaw,BKP,BPW,Bill_2009,DGJM,Ewa:Delfim:JOTA}.
The mathematics of time scales was born in 1988 \citep{MR1062633},
providing a rich calculus that unifies and extends the theories
of difference and differential equations \citep{Bh}. A time scale is a model of time.
Besides the standard cases of the whole real line (continuous-time case)
and all integers (discrete-time case) there are many other models of time included,
\textrm{e.g.}, the time scale $\mathbb{P}_{a,b} = \bigcup_{k=0}^{\infty} [k(a+b),k(a+b)+a]$,
$q$-scales, quantum time scales (objects with nonuniform domains),
and many others -- see \citep{Bh,Bh_adv}. However, the discrete-time systems
on time scales are based on the difference operator and not on the more conventional shift operator.
Note that the difference operator description provides a smooth transition from sampled-data
algorithms to their continuous-time counterparts \citep{GGS01}. In order to deal with
non-traditional applications in areas such as medicine, economics, or engineering,
where the system dynamics are described on a time scale partly continuous and partly discrete,
or to accommodate non-uniform sampled systems, one needs to work with systems
defined on a time scale -- see, \textrm{e.g.}, \citep{MR2218315,MR2433734}.

The study of control systems defined on an arbitrary time scale
is a six years old emerging research area under strong current research
\citep{D2,D1,BPaw,BKP,BPW,Bill_2009,MR2484109,MR2436490,MR2394394,DGJM,KBPW:09,Ewa:Delfim:JOTA},
motivated by multidisciplinary applications that require simultaneous modeling
of discrete and continuous data \citep{Seiffertt}.
In \citep{BPaw} the question of realizability
of linear time-invariant control systems defined on time scales
is studied. Main result shows how to construct a state space representation
of an abstract input/output map and gives conditions
for this map to allow such a representation. It is also proved
that classical Kalman conditions \citep{Kalman} are still valid for systems on time scales.
The assumption of regressivity for the considered control systems is dropped.
This assumption implies existence and uniqueness of both forward and backward solutions
of linear delta differential equations \citep{Bh}.
In problems that are studied in \citep{BPaw} only forward solutions
are needed and they exist without the regressivity hypothesis.
In \citep{DGJM} it is developed, under the regressivity assumption,
the foundational notions of controllability, observability,
and realizability of time-varying linear control systems
defined on an arbitrary time scale. The proposed generalized framework
has already shown promising applications \citep{DGJM}. A delta-NARX model has been
suggested for modeling nonlinear control systems,
and it has been applied to the identification of a van der Pol oscillator \citep{AK}.

The theory of time scales is, however, not unique,
and two approaches are followed in the literature:
one dealing with the delta calculus (the forward approach);
the other dealing with the nabla calculus (the backward approach) \citep{MR1962545}.
Available results on linear control systems on time scales are essentially
restricted to the forward approach, but recent applications in economics
have suggested that the backwards framework is sometimes more natural
and preferable \citep{MR2218315,MR2433734,MR2436490,MR2528200,Nat:Del}.
This becomes evident when one considers that the
time scales analysis can also have important
implications for numerical analysts, who often prefer
backward differences rather than forward differences to
handle their computations due to practical implementation reasons
and also for better stability properties of implicit discretizations.

The goal of this paper is to develop the foundations
of a backward linear control systems theory
on an arbitrary time scale. For that we make use of the recent duality theory
\citep{Caputo}, which presents tools for obtaining nabla results
from the delta calculus and viceversa,
without making any assumptions on the regularity of the time scales
(thus diverging from the approach in \citep{GGS05}).
The organization of the paper is as follows.
Section~\ref{sec:dual} presents the main definitions and concepts of duality
on time scales. In Section~\ref{sec:lcs} we prove existence and uniqueness of a backward solution
for time-varying linear control systems. In sections \ref{sec:cont} and \ref{sec:obs}
we show that controllability and observability rank conditions are still valid
for time-invariant and time-varying linear control systems defined on backward (dual) time scales.
Finally, in Section~\ref{sec:real} we prove conditions of existence
of minimal realizations for the considered backward systems.
We end with Section~\ref{sec:conc} of conclusions.


\section{Duality}
\label{sec:dual}

We assume the reader to be familiar with the
calculus on time scales \citep{MR1062633,Bh}.
Let $\T$ be an arbitrary time scale and let $\T^\star:=\{s\in\R:-s\in\T\}$.
Note that $\T^\star$ is a nonempty closed subset of the real line (so it is
also a time scale), and that the map $\xi:\T\rar\T^*$
defined by $\xi(t) = - t$ is onto and one-to-one.
The new time scale $\T^\star$ is called the
\emph{dual time scale} (to $\T$). It follows that
$(\T^\kappa)^\star=(\T^\star)_\kappa$ and $(\T_\kappa)^\star=(\T^\star)^\kappa$.
By $[a,b]_\T$ we denote the intersection of the real
interval $[a,b]$ with the time scale $\T$, \textrm{i.e.},
$[a,b]_\T := [a,b] \cap \T$. Similarly for $\T^\star$.

The \emph{dual function} to $f:\T\rar\R$, defined on $\T^\star$,
is the function $f^\star:\T^\star\rar\R$ given by $f^\star(s):=f(-s)$ for all $s\in\T^\star$.
It can be shown that $f$ is rd-continuous (resp. ld-continuous) if and only if
its dual function $f^\star$ is ld-continuous (resp. rd-continuous)
\citep{Caputo}.

Given $\sigma,\rho:\T\rar\T$, the jump operators of the time scale $\T$,
then the jump operators for time scale $\T^\star$,
$\hat{\sigma}, \hat{\rho}:\T^\star\rar\T^\star$,
are given by the following relations \citep{Caputo}:
\begin{equation}
\label{hat_rho}
\begin{split}
\hat{\sigma}(s) &= -\rho(-s)=-\rho^\star(s)\, , \\
\hat{\rho}(s) &= -\sigma(-s)=-\sigma^\star(s) \, ,
\end{split}
\end{equation}
for all $s\in\T^\star$. These two equalities implies that the forward graininess
$\mu:\T\rar [0,\infty)$ and the backward graininess
$\hat{\nu} :\T^\star \rar [0,\infty)$ are related by
\begin{equation*}
 \hat{\nu}(s)=\mu^\star(s)\;\;\;\;\;\text{for\;all\;$s\in\T^\star.$}
\end{equation*}
Similarly, the backward graininess $\nu:\T\rar [0,\infty)$
and the forward graininess
$\hat{\mu} :\T^\star\rar [0,\infty)$ are related by
\begin{equation}
\label{hat_mu}
 \hat{\mu}(s)=\nu^\star(s)\;\;\;\;\;\text{for\;all\;$s\in\T^\star.$}
\end{equation}

\begin{lemma}\citep{Caputo}
\label{delta_star}
Let function $f:\T\rar\R$ be delta (resp. nabla) differentiable at point
$t_0\in\T^\kappa$ (resp. at $t_0\in\T_\kappa$). Then function $f^\star:\T^\star\rar\R$
is nabla (resp. delta) differentiable at $-t_0\in(\T^\star)_\kappa$
(resp. at $-t_0\in(\T^\star)^\kappa$) and the following relations
hold true:
 \[f^\Delta(t_0)=-(f^\star)^{\hat{\nabla}}(-t_0)\;\;\;
 (\text{resp.}\;\;\;f^{\nabla}(t_0)=-(f^\star)^{\hat{\Delta}}(-t_0)),\]
 or
 \[f^\Delta(t_0)=-((f^\star)^{\hat{\nabla}})^\star(t_0)\;\;\;
 (\text{resp.}\;\;\;f^{\nabla}(t_0)=-((f^\star)^{\hat{\Delta}})^\star(t_0)),\]
 or
 \[(f^\Delta)^\star(-t_0)=-((f^\star)^{\hat{\nabla}})(-t_0)\;\;\;
 (\text{resp.}\;\;\;(f^{\nabla})^\star(-t_0)=-(f^\star)^{\hat{\Delta}}(-t_0)).\]
\end{lemma}

Additionally, from properties of the $\Delta$ derivative on the time scale $\T$
and from (\ref{hat_mu}), it follows that for any nabla differentiable function
$f:\T\rar\R$ its dual function $f^\star:\T^\star\rar\R$ is delta differentiable with
\[(f^\star)^{\hat{\sigma}}(s)=f^\star(s)+\hat{\mu}(s) \cdot
(f^\star)^{\hat{\Delta}}(s) \, \quad \text{for all } s \in (\T^\star)^\kappa \, .\]
Thus, $f^\rho(t)=f(t)-\nu(t) \cdot f^{\nabla}(t)$  for all  $t \in \T_\kappa$.

\begin{proposition}\citep{Caputo}
\label{int_star}
 \begin{enumerate}
   \item [(i)] Let $f:[a,b]_{\T}\rar\R$ be a rd-continuous function. Then,
   \[\int_a^bf(t)\Delta t=\int_{-b}^{-a}f^\star(s)\hat{\nabla} s \, .\]
   \item [(ii)] Let $f:[a,b]_{\T}\rar\R$ be a ld-continuous function. Then,
   \[\int_a^b f(t)\nabla t = \int_{-b}^{-a}f^\star(s)\hat{\Delta} s \, .\]
 \end{enumerate}
\end{proposition}


\section{Linear control systems}
\label{sec:lcs}

Let us consider a time-varying system defined on a given time scale $\T$:
\begin{equation}
\label{def_trans}
x^\Delta(t)=A(t)x(t)
\end{equation}
with $t \in \T^\kappa$, $t\geq t_0$, $t_0\in \T$, $A(t)\in\R^{n\times n}$.
Recall that by \emph{transition function} one means the unique forward solution
of the system (\ref{def_trans}) with initial condition $x(t_0)=I$,
where $I$ denotes the identity matrix $I\in\R^{n\times n}$.
Its value at point $t\in\T$ is denoted
by $\Phi_A(t,t_0)$. When $A$ is time invariant, we denote the solution
of (\ref{def_trans}) with initial condition $x(t_0)=I$
by $e_A(t,t_0)$, and call it the \emph{exponential matrix function}.
There are important distinctions between the two notations,
as $\Phi_A(t,t_0)=e_A(t,t_0)$ if and only if $A$ is a constant matrix.

By the \emph{backward system} (to the system (\ref{def_trans})) we mean
\[y^{\hat{\nabla}}(s)=\bar{A}(s)y(s) \, , \]
defined on the dual time scale $\T^\star$
with $s \le s_0 = - t_0$, $s_0 \in \T^\star$,
$s\in(\T^\star)_\kappa$, and $\bar{A}(s):=-A^{\star}(s)$.

\begin{proposition}
\label{lin_aut}
Let $\T^\star$ be a time scale. Then the time-varying system
\begin{equation}
\label{def_trans:d}
y^{\hat{\nabla}}(s)=\bar{A}(s)y(s),\;\;\;y(s_0)=y_0 \, ,
\end{equation}
where $s_0 \in \T^\star$, $s \in(\T^\star)_\kappa$, and $\bar{A}\in\R^{n\times n}$,
has a unique backward solution. This solution is of the form $y(s)=\Psi_{\bar{A}}(s,s_0)y_0$
for any $s\leq s_0$, where $\Psi$ denotes the transition function dual to $\Phi$, \textrm{i.e.},
$\Psi_{\bar{A}}(s,s_0) = (\Phi_{-\bar{A}^\star})^{\star}(-s,-s_0)$.
\end{proposition}

\begin{proof}
Let $t_0 := -s_0$ and $t:=-s$ for any $s\in(\T^\star)_\kappa$.
From Lemma~\ref{delta_star}
it follows that the system (\ref{def_trans}) defined
on $\T$ (the dual of $\T^\star$) with initial condition $x(t_0)=y_0$
can be rewritten on $\T^\star$ as
\begin{equation}
\label{IVP_T*}
y^{\hat{\nabla}}(s)=-A^\star(s)y(s),\;\;\;y(s_0)=y_0 \, ,
\end{equation}
where $y$ is the dual vector function to $x$, \textrm{i.e.}, $y(s)=x^\star(-s)$.
Because system (\ref{def_trans}) with initial condition $x(t_0)=y_0$
has a unique forward solution on $\T$ \citep{Bh,Bill_2009},
then system (\ref{IVP_T*}) has also a unique backward solution
$y(s)=\Psi_{-A^\star}(s,s_0)y_0$ on $\T^\star$.
\end{proof}

It follows from Proposition~\ref{lin_aut}
that the solutions of (\ref{def_trans})
and (\ref{def_trans:d}) are related by
$x(t) = y^\star(-t)$ and $y(s) = x^\star(-s)$.

\begin{remark}
In the case of a time-invariant system \eqref{def_trans},
\textrm{i.e.}, when $A$ is a constant matrix, we have
$e_A(t,t_0)=\left(\hat{e}_{\bar{A}}\right)^\star(-t,-t_0)$
with $\hat{e}_{\bar{A}}$ the nabla exponential function
on $\T^\star$ for the constant matrix $\bar{A}=-A^{\star}$.
\end{remark}

Let us consider now a time-varying control system
on the time scale $\T$:
\begin{equation}
\label{nonh_T}
x^\Delta(t)=A(t)x(t)+B(t)u(t),\;\;\;x(t_0)=x_0,
\end{equation}
where $u(\cdot)$ is a rd-continuous control taking values
in $\R^m$, and $A(t)\in\R^{n\times n}$ and $B(t)\in\R^{n\times m}$
are time-dependent matrices defined on $\T^\kappa$.

\begin{proposition}
\label{cs:lin_aut}
Let $\T^\star$ be the dual time scale to $\T$,
with $\T$ the time scale where the control system (\ref{nonh_T})
is defined. Then the control system
 \begin{equation}
\label{nonh_T:d}
  y^{\hat{\nabla}}(s)=\bar{A}(s)y(s)+\bar{B}(s)v(s),\;\;\;y(s_0)=x_0,
 \end{equation}
where $s_0 = - t_0$, $s \le s_0$, $s \in (\T^\star)_\kappa$,
$\bar{A}(s) = -A^\star(s)$, and $\bar{B}(s)=-B^\star(s)$,
has a unique backward solution. The solution is given by
\[y(s)=\Psi_{\bar{A}}(s,s_0)x_0+\int_s^{s_0}\Psi_{\bar{A}}(s,-\hat{\rho}(\varsigma))
\bar{B}(\varsigma)v(\varsigma)\hat{\nabla}\varsigma\]
for any $s\leq s_0$, where $\Psi$ is the transition function dual to $\Phi$,
\textrm{i.e.}, $\Psi_{\bar{A}}(s,s_0) = (\Phi_A)^{\star}(-s,-s_0)$.
\end{proposition}
\begin{proof}
Let $s:=-t$ for any $t\in\T^\kappa$.
So $s\in (\T^\star)_\kappa$ and $s\leq s_0$. The control system
\eqref{nonh_T} defined on $\T$
 can be rewritten on the time scale $\T^\star$ as
\begin{equation}\label{control_sys_T*}
   y^{\hat{\nabla}}(s)=\bar{A}(s)y(s)+\bar{B}(s)v(s) \, ,
\end{equation}
where $y$ and $v$ are the dual vector to $x$ and $u$,
respectively, \textrm{i.e.}, $y(s)=x^\star(-s)$
and $v(s)=u^\star(-s)$. Equation (\ref{nonh_T}) has a unique forward solution on $\T$
\citep{Bh,Bill_2009} given by
\[x(t)=\Phi_A(t,t_0)x_0 + \int_{t_0}^{t}\Phi_A(t,\sigma(\tau))B(\tau)u(\tau)\Delta\tau \, .\]
Thus, by definition of dual time scale and by (\ref{hat_rho}), the dual control system
(\ref{control_sys_T*}) on $\T^*$ has also a unique, but backward, solution
\[y(s)=\Psi_{-A^\star}(s,s_0)x_0-\int_s^{s_0}\Psi_{-A^\star}(s,
-\hat{\rho}(\varsigma))B^\star(\varsigma)u^\star(\varsigma)\hat{\nabla}\varsigma.\]
\end{proof}

It follows from Proposition~\ref{cs:lin_aut}
that state and control variables of systems (\ref{nonh_T})
and (\ref{nonh_T:d}) are related by
$\left(x(t),u(t)\right) = \left(y^\star(-t),v^\star(-t)\right)$
and $\left(y(s),v(s)\right) = \left(x^\star(-s),u^\star(-s)\right)$.

\begin{remark}
In the particular case when $\bar{A}$
and $\bar{B}$ are constant matrices, then
the solution of (\ref{nonh_T:d}) given
in Proposition~\ref{cs:lin_aut} takes the form
\[y(s)=\hat{e}_{\bar{A}}(s,s_0)x_0
+\int_s^{s_0}\hat{e}_{\bar{A}}(s,-\hat{\rho}(\varsigma))\bar{B}
v(\varsigma)\hat{\nabla}\varsigma\]
for any $s\leq s_0$.
\end{remark}


\section{Controllability}
\label{sec:cont}

Let $\T^\star$ be a given time scale with operators $\hat{\sigma}$, $\hat{\rho}$,
$\hat{\mu}$, $\hat{\nu}$, $\hat{\Delta}$, and $\hat{\nabla}$.
Let us consider the following linear time-varying system defined on $\T^\star$:
\begin{equation*}
 \begin{split}
  \bar{\Lambda}:\;\;\; y^{\hat{\nabla}}(s)=\bar{A}(s)y(s)+\bar{B}(s)v(s)\\
                  \gamma(s)=\bar{C}(s)y(s)+\bar{D}(s)v(s)
  \end{split}
\end{equation*}
with initial condition $y(s_0)=y_0$, $s\leq s_0$,
where $y(s)\in\R^n$ is the state of the system at time $s$,
$v(s)\in \R^m$ is the control value at time $s$,
and $\bar{A}(s)\in\R^{n\times n}$, $\bar{B}(s)\in\R^{n\times m}$, $\bar{C}(s)\in\R^{p\times n}$,
and $\bar{D}(s)\in\R^{p\times m}$, $p,m\leq n$, are ld-continuous as functions of $s$.

We say that system $\bar{\Lambda}$ is \emph{controllable} if for any two states
$y_0\in\R^n$ and $y_1\in\R^n$ there exist $s_0,\;s_1\in\T^\star$,  $s_1<s_0$,
and a piecewise ld-continuous control $v(s)$, $s\in[s_1,s_0]_{\T^\star}$,
such that for $y_0=y(s_0)$ one has  $y(s_1)=y_1$.
The set of all points that can be reachable from the point $y_0=y(s_0)$ in time
$s_1\in\T^\star$ is called the \emph{reachable set} from $y_0$
in time $s_1$ and is denoted by $\mathcal{\bar{R}}_{y_0}(s_1,s_0)$.
The set of all points reachable from $y_0=y(s_0)$ in finite time
$s\in\T^\star$ will be denoted by $\mathcal{\bar{R}}_{y_0}(s_0)$.
Note that $\mathcal{\bar{R}}_{y_0}(s_1,s_0)=\Psi_{\bar{A}}(s_1,s_0)+\mathcal{\bar{R}}_{0}(s_1,s_0)$.

Let us assume that $\T^\star$ is a time scale for which $\hat{\rho}$
is sufficiently ld-continuous $\hat{\nabla}$ differentiable.
Define the sequence of matrices
\begin{equation}
\label{K}
\bar{K}_j(s):=-\left.\frac{\partial^j}{\hat{\nabla} z^j}\left[\Psi_{\bar{A}}(\hat{\rho}(s),
-\hat{\rho}(z))\bar{B}(z)\right]\right|_{z=s} \, \quad j = 0, 1, 2, \ldots
\end{equation}

\begin{theorem}\label{controllability_tv}
Let $r$ be a positive integer such that $\bar{B}(\cdot)$
is $r$-times ld-continuously $\hat{\nabla}$ differentiable and both
$\hat{\rho}(\cdot)$ and $\bar{A}(\cdot)$ are ld-continuously $r-1$-times $\hat{\nabla}$ differentiable
on $[s_0,s_1]_{\T^\star}$. Then the linear system $\bar{\Lambda}$ is controllable
on $[s_1,s_0]_{\T^\star}$ if for some $s_c\in(s_1,s_0]_{\T^\star}$
the matrix \[(\bar{K}_0(s_c)\;\;\bar{K}_1(s_c)\;\ldots\;\bar{K}_r(s_c))\]
is of full rank, where $\bar{K}_j$, $j=0,1,\ldots,r$, are the matrices given by (\ref{K}).
\end{theorem}
\begin{proof}
Let $t := -s$ (and $t_0 :=-s_0$) so that $t$ belongs to $\T$ -- the dual time scale of $\T^\star$ --
with operators $\sigma$, $\rho$, $\mu$, $\nu$, $\Delta$, and $\nabla$.
Using Lemma~\ref{delta_star} we can rewrite system $\bar{\Lambda}$ on $\T$ as
\begin{equation}
\label{sys_T}
\begin{split}
\Lambda:\;\;\; x^\Delta(t) &= A(t)x(t)+B(t)u(t)\, , \quad x(t_0) =y_0\, ,\\
z(t) &= C(t)x(t)+D(t)u(t)\, ,
\end{split}
\end{equation}
where $A(t)\in\R^{n\times n}$, $B(t)\in\R^{n\times m}$, $C(t)\in\R^{p\times n}$,
and $D(t)\in\R^{p\times m}$, $p,m\leq n$, are rd-continuous as functions of $t$.
System $\Lambda$ is controllable on an interval $[t_0,t_1]_{\T}$ if for some $t_c\in[t_0,t_1)_{\T}$
\[{\rm rank}(K_0(t_c)\;\;K_1(t_c)\;\ldots\;K_r(t_c))=n\]
with matrices $K_j(s):=-\frac{\partial^j}{\Delta^j}\left[\Phi_A(\sigma(t),\sigma(s))B(s)\right]_{|_{s=t}}$,
$j=0,1,\ldots,r$. This fact was proved in \citep{DGJM} for a regressive system $\Lambda$.
The proof is still valid without the regressivity assumption.
The dual system to $\Lambda$ is $\bar{\Lambda}$ where $y$ is the dual vector to $x$,
$\bar{A}(s):=-A^\star(s)$, $\bar{B}(s):=-B^\star(s)$, $\bar{C}(s):=C^\star(s)$,
$\bar{D}(s):=D^\star(s)$ and $A^\star$, $B^\star$, $C^\star$, and $D^\star$
are the dual function matrices to $A$, $B$, $C$ and $D$, respectively. Moreover,
from the form of matrices $K_j$ and Lemma~\ref{delta_star}, it follows that the dual matrix $K^\star$ is of form
\[K_j^\star(-t)=-\left.\frac{\partial^j}{\hat{\nabla} s^j}\left[
\Phi^*_{-A^*}\left(\hat{\rho}(s),-\hat{\rho}(z)\right)B^\star(z)\right]\right|_{z=s},\]
where $\Phi^\star_{-A^*}$ is the dual matrix to $\Phi_A$.
We obtain (\ref{K}) by putting $\bar{K}=K^\star$
and $\Psi_{\bar{A}}=\Phi^\star_{-A^*}$.
\end{proof}


\begin{example}
  Let us consider the control system
  \begin{equation}
\label{eq:ex1}
   \begin{cases}
    y_1^{\hat{\nabla}} = y_1+s^2y_2+v_1\\
    y_2^{\hat{\nabla}} = -y_2-s v_2
   \end{cases}
  \end{equation}
  defined on the time scale $\T^\star$ dual of $\T=\bigcup_{k\in\Z} [2k,2k+1]$. Because function
  \[\Psi_{\bar{A}}=
   \left[ \begin {array}{cc} {2}^{k-l}{{\rm e}^{s-z}}&{2}^{k-l}{s}^{2}
 \left( s-z \right) {{\rm e}^{s-z}}+{2}^{k-l-1}{s}^{2} \left( k-l
 \right) {{\rm e}^{s-z}}\\\noalign{\medskip}0&{2}^{k-l}{{\rm e}^{s-z}}
\end {array} \right]
  \]
  is the transition matrix of this system, then
  \begin{equation*}
   \begin{split}
    K_0&=\left(
         \begin{array}{c}
           1 \\
           -s \\
         \end{array}
       \right) \, ,\\
     K_1&=-\frac{\partial}{\hat{\nabla}z}
  \left.\left[ \begin {array}{c} {2}^{k-l}{{\rm e}^{s-z}}- \left( {2}^{k-l}{s
}^{2} \left( s-z \right) {{\rm e}^{s-z}}+{2}^{k-l-1}{s}^{2} \left( k-l
 \right) {{\rm e}^{s-z}} \right) s\\\noalign{\medskip}-s{2}^{k-l}{
{\rm e}^{s-z}}\end {array}
       \right]\right|_{z=s}.
   \end{split}
  \end{equation*}
For any $s\in(2l-1,z]$, $s\neq 0$,
${\rm rank}\left(
\begin{array}{cc}
K_0 & K_1 \\
\end{array}
\right)=2$.
If $s=2l-1$, then for $z\in(2l-3,2l-2]$
${\rm rank}\left(
               \begin{array}{cc}
                 K_0 & K_1 \\
               \end{array}
             \right)=2 \Leftrightarrow s\neq 0$.
We conclude that system \eqref{eq:ex1} is controllable on $\R^2\setminus \{0\}$.
\end{example}


Let us consider now a time-invariant system:
\begin{equation}
\label{ukl_lin_ti_1}
\begin{split}
y^{\hat{\nabla}}(s) &= \bar{A}y(s)+\bar{B}v(s)\\
  \gamma(s) &= \bar{C}y(s)
\end{split}
\end{equation}
with initial condition $y(s_0)=y_0$, $s\leq s_0$,
and $\bar{A}\in\R^{n\times n}$, $\bar{B}\in\R^{n\times m}$
and $\bar{C}\in\R^{p\times n}$ constant matrices.

\begin{theorem}
\label{thm:4.3}
Let us assume that the interval $[s_1,s_0]_{\T^\star}$ consists of at least $n+1$ elements.
The following conditions are equivalent:
\begin{enumerate}
  \item [(i)] $\mathcal{\bar{R}}_{y_0}(s_1,s_0)=\R^n$;
  \item [(ii)] ${\rm rank}(\bar{P}_0\bar{B},\bar{P}_1\bar{B},
  \ldots,\bar{P}_{n-1}\bar{B})=n$ where matrices $\bar{P}_i$, $i=0,1,\ldots,n-1$, are given recursively by
  $\bar{P}_0=I$ and $\bar{P}_{k+1}=(\bar{A}-\lambda_{k+1}I)\bar{P}_k$,
  $k=0,1,\ldots,n-1$, with $\lambda_1,\ldots,\lambda_n$ the eigenvalues of the matrix $\bar{A}$;
  \item [(iii)] ${\rm rank}(\bar{B},\bar{A}\bar{B},\ldots,\bar{A}^{n-1}\bar{B})=n$;
  \item[(iv)] the system (\ref{ukl_lin_ti_1}) is controllable.
\end{enumerate}
\end{theorem}

\begin{proof}
Let $t:=-s$ for any $s\in\T^\star$ and let $t_0:=-s_0$,
so that $t$ is an element of the dual time scale $\T$ (with operators $\sigma$, $\rho$,
$\mu$, $\nu$, $\Delta$, $\nabla$). Using Lemma~\ref{delta_star} we write
the dual system on $\T$:
\begin{equation}
\label{sys_ti1_T}
\begin{split}
x^\Delta(t) &= Ax(t)+Bu(t)\, ,\\
z(t) &= Cx(t)\, ,\\
x(t_0)&=y_0
\end{split}
\end{equation}
with $A=-\bar{A}$, $B=-\bar{B}$, and $C=\bar{C}$.
Since conditions
\begin{enumerate}
  \item [(a)] $\mathcal{R}_{y_0}(t_0,t_1)=\R^n$;
  \item [(b)] ${\rm rank}(P_0B,P_1B,\ldots,P_{n-1}B)=n$
  where matrices $P_i$, $i=0,1,\ldots,n-1$, are given recursively by
  $P_0=I$ and $P_{k+1}=(A-\lambda_{k+1}I)P_k$,
  $k=0,1,\ldots,n-1$, with $\lambda_1,\ldots,\lambda_n$ the eigenvalues of the matrix $A$;
  \item [(c)] ${\rm rank}(B,AB,\ldots,A^{n-1}B)=n$;
  \item [(d)] the system given by (\ref{sys_ti1_T}) is controllable;
\end{enumerate}
are equivalent on the time scale $\T$ (for the proof see \citep{BPaw}), then,
after coming back to the time scale $\T^\star$
we see that conditions $(i)$--$(iv)$ are also equivalent on the time scale $\T^\star$.
\end{proof}


\begin{example}
\label{ex1}
Let us consider the control system
\begin{equation}
\label{eq:ex2}
\begin{cases}
  y_1^{\hat{\nabla}}=y_2\\
  y_2^{\hat{\nabla}}=-3y_1-4y_2+v \, .
\end{cases}
\end{equation}
Independently of the time scale, we have
always ${\rm rank}(\bar{B},\bar{A}\bar{B})=2$. Thus the system \eqref{eq:ex2}
is controllable by items (iii) and (iv) of Theorem~\ref{thm:4.3}.
\end{example}


\section{Observability}
\label{sec:obs}

Let us consider the linear time-varying control system
\begin{equation}\label{ukl_lin_1}
\begin{split}
  \bar{\Lambda}:\;\;\; y^{\hat{\nabla}}(s)&=\bar{A}(s)y(s)+\bar{B}(s)v(s)\\
\gamma(s)&=\bar{C}(s)y(s) \, ,
\end{split}
\end{equation}
$s\leq s_0$, defined on a given time scale $\T^\star$
under a given initial condition $y(s_0)=y_0$. We say that such a system
$\bar{\Lambda}$ is \emph{observable on $[s_1,s_0]_{\T^\star}$}
if any initial state $y(s_0)=y_0$ can be uniquely determined
by the output function $\gamma(s)$ for $s\in(s_1,s_0]_{\T^\star}$.

Let us assume that the time scale $\T^\star$ is such that $\hat{\rho}$ is sufficiently
$\hat{\nabla}$ differentiable with ld-continuous derivatives.
Define the sequence of matrices
\begin{equation}
\label{L}
\bar{L}_j(s):=-\left.\frac{\partial^j}{\hat{\nabla} z^j}\left[
\bar{C}(s)\Psi_{\bar{A}}(s,z)\right]\right|_{z=s} \, , \quad j = 0, 1, 2, \ldots
\end{equation}

Using a similar reasoning as in Theorem~\ref{controllability_tv} we have the following:

\begin{theorem}
 Let $r$ be a positive integer such that the matrix function $\bar{C}(s)$ is $r$-times ld-continuously
 $\hat{\nabla}$ differentiable and both $\hat{\rho}(s)$ and $\bar{A}(s)$
 are $r-1$-times ld-continuously $\hat{\nabla}$ differentiable for any $s\in[s_0,s_1]_{\T^\star}$.
 Then the linear system (\ref{ukl_lin_1}) is observable
 on $[s_1,s_0]_{\T^\star}$ if for some $s_c\in(s_1,s_0]_{\T^\star}$ the matrix
 \[\left(
     \begin{array}{c}
       \bar{L}_0(s_c) \\
       \bar{L}_1(s_c) \\
       \vdots \\
       \bar{L}_r(s_c)\\
     \end{array}
   \right)\]
is of full rank, where $\bar{L}_j$ are the matrices given by (\ref{L}), $j=0,1,\ldots,r$.
\end{theorem}

Let us consider the particular case of a time invariant system
(\ref{ukl_lin_ti_1}) on the time scale $\T^\star$:

\begin{theorem}
Assume that the interval $[s_1,s_0]_{\T^\star}$ consists
of at least $n+1$ points. The following conditions are equivalent:
\begin{enumerate}
  \item [(i)] system (\ref{ukl_lin_ti_1}) is observable;
  \item [(ii)] ${\rm rank}\left(
                    \begin{array}{c}
                      \bar{C}\bar{P}_0\\
                      \bar{C}\bar{P}_1\\
                      \vdots\\
                      \bar{C}\bar{P}_{n-1}
                    \end{array}
                  \right)=n$, where matrices $\bar{P}_i$, $i=0,1,\ldots,n-1$,
                  are given recursively by $\bar{P}_0=I$, $\bar{P}_{k+1}=(\bar{A}-\lambda_{k+1}I)\bar{P}_k$,
                  $k=0,\ldots,n-1$, and $\lambda_1,\ldots,\lambda_n$ are the eigenvalues of $\bar{A}$;
  \item [(iii)] ${\rm rank}\left(
                    \begin{array}{c}
                      \bar{C} \\
                      \bar{C}\bar{A} \\
                      \vdots \\
                      \bar{C}\bar{A}^{n-1} \\
                    \end{array}
                  \right)=n$.
\end{enumerate}
\end{theorem}

\begin{proof}
Without loss of generality we may assume that $\bar{B}(s)=0$.
Putting $t:=-s$ for any $s\in\T^\star$ and letting $t_0:=-s_0$,
we rewrite the system from the time scale $\T^\star$ onto $\T$: using
Lemma~\ref{delta_star} we can rewrite the given system (\ref{ukl_lin_ti_1})
in the form (\ref{sys_ti1_T}). Since conditions,
\begin{enumerate}
\item [(a)] system (\ref{sys_ti1_T}) is observable;
\item [(b)] ${\rm rank}\left(
                    \begin{array}{c}
                      CP_0\\
                      CP_1\\
                      \vdots\\
                      CP_{n-1}
                    \end{array}
                  \right)=n$, where matrices $P_i$, $i=0,\ldots,n-1$, are given recursively by
$P_0=I$ and $P_{k+1}=(A-\lambda_{k+1}I)P_k$,
$k=0,\ldots,n-1$, with $\lambda_1,\ldots,\lambda_n$ the eigenvalues of $A$;
\item [(c)] ${\rm rank}\left(
                    \begin{array}{c}
                      C \\
                      CA \\
                      \vdots \\
                      CA^{n-1} \\
                    \end{array}
                  \right)=n$;
\end{enumerate}
are equivalent on the time scale $\T$ (see \citep{BPaw}),
taking $s_0:=-t_0$ and $s:=-t$ for any $t\in\T$ we obtain
equivalence of conditions $(i)$--$(iii)$ on the time scale $\T^\star$.
\end{proof}


\section{Realizability}
\label{sec:real}

One can notice that the control $v(\cdot)$ and the output $\gamma(\cdot)$
of system $\bar{\Lambda}$ given by (\ref{ukl_lin_1})
are related in the following way:
\[\gamma(s)=C(s)\Psi_{\bar{A}}(s,s_0)y_0+\int_s^{s_0}C(s)\Psi_{\bar{A}}(s,-\hat{\rho}(z))
\bar{B}(z)v(z)\hat{\nabla}z\]
for $s\leq s_0$, $s_0\in\T^\star$ fixed. The operator
$\bar{\mathcal{A}}[v(\cdot)] := \int_s^{s_0}\bar{G}(s,z)v(z)\hat{\nabla}z$,
where
\[\bar{G}(s,z)=C(s)\Psi_{\bar{A}}(s,-\hat{\rho}(z))
\bar{B}(z),\]
is called the \emph{action} of $\bar{\Lambda}$, while function $\bar{G}$ is called
the \emph{weighting pattern} of the system. Note that different systems
of form (\ref{ukl_lin_1}) can define the same weighting pattern.
Each of them is called a \emph{realization}. A realization is \emph{minimal}
if no realization of $\bar{G}(s,z)$ with dimension less than $n$ exists
($n$ is the dimension of matrix $\bar{A}$).

\begin{theorem}
\label{real_T}
The weighting pattern $\bar{G}(s,z)$ is realizable if and only if for each
$s_1\in\T^\star$, $s_1\leq s_0$, there exists a ld-continuous matrix function
$\bar{H} : (-\infty,s_1]_{\T^\star}\rar\R^{q\times n}$
and a ld-continuous matrix function
$\bar{F} : (-\infty,s_1]_{\T^\star}\rar\R^{n \times r}$ such that
\[\bar{G}(s,z)=\bar{H}(s)\bar{F}(z).\]
\end{theorem}

\begin{proof}
Taking $t := -s$ and $\tau=-z$ (so that $t,\tau\in\T$, where $\T$
is the dual time scale of $\T^\star$ with operators $\sigma$, $\rho$,
$\mu$, $\nu$, $\Delta$, and $\nabla$) and using Lemma~\ref{delta_star},
we can rewrite system $\bar{\Lambda}$ (\ref{ukl_lin_1})
into the system $\Lambda$ (\ref{sys_T}) on the time scale $\T$.
A forward characterization of realizable systems on $\T$
can be obtained from \citep{DGJM}. Note that if $G^\star$
is the dual function to $G$ and $H^\star$, $F^\star$ are dual
to $H$ and $F$, respectively (note that $H^\star$ and $F^\star$ are ld-continuous),
then $G^\star(-t,-\sigma(-\tau))=H^\star(-t)F^\star(-\sigma(-t\tau)$.
On the other hand, we have the weighting pattern
\begin{equation}
\label{weighting_pat}
G^\star(-t,-\sigma(-\tau))=C^\star(-t)\Phi_{-A^\star}\left(-t,
-\sigma(-\tau)\right)(-B^\star(\tau)).
\end{equation}
It means that for $\bar{G}=G^\star$, $\bar{H}=H^\star$, and $\bar{F}=F^\star$,
the weighting pattern $\bar{G}(s,-\hat{\rho}(z))$ is realizable
if and only if $\bar{G}(s,-\hat{\rho}(z))=\bar{H}(s)\bar{F}(-\hat{\rho}(z))$,
where $\bar{G}(s,-\hat{\rho}(z))=C(s)\Psi_{\bar{A}}(s,-\hat{\rho}(z))\bar{B}(z)$.
\end{proof}

\begin{remark}
For time-invariant systems the weighting pattern is given by
$\bar{G}(s,-\hat{\rho}(z))=C\Psi_{\bar{A}}(s,-\hat{\rho}(z))\bar{B}$.
The backward characterization of realizable time-invariant systems on
time scales follows from Theorem~\ref{real_T},
and the forward characterization of realizable time-invariant systems
on time scales is given in \citep{BPaw}.
\end{remark}

\begin{definition}
We say that control system $\bar{\Lambda}$ given by \eqref{ukl_lin_1}
is \emph{progressive} if matrix $I - \hat{\nu}(s) \bar{A}(s)$ is invertible.
\end{definition}

\begin{remark}
If $\bar{\Lambda}$ is progressive,
then its dual $\bar{\Lambda}^\star = \Lambda$ is regressive -- see
\citep{Bh} for the definition of regressivity.
\end{remark}

\begin{theorem}
Let us assume that the control system
$\bar{\Lambda}$ given by \eqref{ukl_lin_1} is progressive.
Let $\bar{\Lambda}$ be a realization of the weighting pattern $\bar{G}(s,-\hat{\rho}(z))$.
Then this realization is minimal if and only if for some $s_0$ and $s_1<s_0$
the system $\bar{\Lambda}$ is both controllable and observable on $[s_1,s_0]_{\T^\star}$.
\end{theorem}

\begin{proof}
Taking $t :=-s$, $\tau=-z$, and $x_0=y_0^\star$ (so that $t,\tau\in\T$ where $\T$
is the dual time scale of $\T^\star$ with operators $\sigma$, $\rho$, $\mu$, $\nu$, $\Delta$, $\nabla$)
and using Lemma~\ref{delta_star} we can convert system $\bar{\Lambda}$
onto system $\Lambda$ defined on $\T$:
 \begin{equation*}
  \begin{split}
   \Lambda:\;\;\; x^\Delta(t) &= A(t)x(t)+B(t)u(t)\, ,\\
     \vartheta(t) &= C(t)x(t),\\
     x(t_0)&=x_0
   \end{split}
  \end{equation*}
where $A(t)\in\R^{n\times n}$, $B(t)\in\R^{n\times m}$, and $C(t)\in\R^{p\times n}$,
$p,m\leq n$, are rd-continuous matricial functions of $t$. Under the regressivity assumption system $\Lambda$
is a realization of the weighting pattern $G(t,\sigma(\tau))$ if and only if for some
$t_0$ and $t_1<t_0$ the system is both controllable and observable on $[t_0,t_1]_{\T}$ \citep{DGJM}.
Now, if $s:=-t$, $z:=-\tau$ and $s_0:=-t_0$ for $t,\tau,t_0\in\T$, then the dual system to
$\Lambda$ is of the form of the system $\bar{\Lambda}$ with $y$ the dual vector to $x$,
$\bar{A}(s):=-A^\star(s)$, $\bar{B}(s):=-B^\star(s)$, and $\bar{C}(s):=C^\star(s)$.
Moreover, from (\ref{weighting_pat}) we can see that the weighting pattern
on the time scale $\T^\star$ is exactly of the form $\bar{G}(s,-\hat{\rho}(z))$.
Since this gives a one-to-one correspondence between the considered systems,
the realization given by $\bar{G}(s,-\hat{\rho}(z))$ is a minimal one
if and only if for some $s_0$ and $s_1<s_0$ the system $\bar{\Lambda}$
is both (backward) controllable and observable on $[s_1,s_0]_{\T^\star}$.
\end{proof}

For the time invariant case we can prove minimal realization
without assuming progressivity.

\begin{theorem}
Let system (\ref{ukl_lin_ti_1}) be a realization
of the weighting pattern $\bar{G}(s,-\hat{\rho}(z))$. This realization
is minimal if and only if for some $s_0$ and $s_1<s_0$ the system (\ref{ukl_lin_ti_1})
is both controllable and observable on $[s_1,s_0]_{\T^\star}$.
\end{theorem}

\begin{proof}
The result follows by applying Caputo's duality \citep{Caputo}
to the control system (\ref{ukl_lin_ti_1}) and using
Theorem~5.7 and Corollary~5.8 in \citep{BPaw}.
\end{proof}


\begin{example}
\label{ex2}
Let us consider a control system
\begin{equation*}
\begin{cases}
 y_1^{\hat{\nabla}}=y_1-y_2+v_1\\
 y_2^{\hat{\nabla}}=2y_2+v_2\\
 \gamma_1=y_1\\
 \gamma_2=-y_2
\end{cases}
\end{equation*}
defined on the time scale $\T^\star$ dual of $\T=\bigcup_{k\in\Z} [2k,2k+1]$.
It is easy to see that this system is both controllable and observable. Because
\[\hat{e}_{\bar{A}}= \left(
                       \begin{array}{cc}
                         2^{k-l}e^{s-z} & 2^{k-l}e^{s-z}-3^{k-l}e^{2(s-z)} \\
                         0 & 3^{k-l}e^{2(s-z)} \\
                       \end{array}
                     \right)
\]
for $k\neq l$, and
\[\hat{e}_{\bar{A}}= \left(
                       \begin{array}{cc}
                         e^{s-z} & e^{s-z}-e^{2(s-z)} \\
                         0 & e^{2(s-z)} \\
                       \end{array}
                     \right)
\]
for $k=l$, then for any $s_1\in[2p-1,2p]_{\T^\star}$ and $k<p<l$
\[\bar{G}(s,z)=\left(
                              \begin{array}{cc}
                                2^{k-p-s_1}e^{s-s_1} & -2\cdot 3^{k-p-s_1}e^{2(s-s_1)} \\
                              \end{array}
                            \right)\left(
                                     \begin{array}{c}
                                       2^{-(k-p+s_1)}e^{-(z-s_1)} \\
                                       3^{s_1}e^{-2(z-s_1)} \\
                                     \end{array}
                                   \right)\, ;
\]
while for any $s_1\in[z,s]_{\T^\star}$ and $k=l$
\[\bar{G}(s,-\hat{\rho}(z))=\left(
                              \begin{array}{cc}
                                2e^{s-s_1} & -2e^{2(s-s_1)} \\
                              \end{array}
                            \right)\left(
                                     \begin{array}{c}
                                       e^{-(z-s_1)} \\
                                       e^{-2(z-s_1)} \\
                                     \end{array}
                                   \right)\, .
\]
\end{example}


\section{Conclusions}
\label{sec:conc}

The \emph{calculus on time scales} has been developed
about 20 years ago by Hilger in order to unify various parallel
results in the theory of discrete and continuous
dynamical systems \citep{MR1066641}.
It found considerable number of applications over the last decade,
particularly in the context of engineering applications
and systems theory and control
\citep{D1,BPW,Bill_2009,MR2394394,Seiffertt,DGJM,KBPW:09,Ewa:Delfim:JOTA}.
The time scales calculus allows two dual formulations:
the delta-calculus where the derivative is the forward difference operator
(yielding the explicit Euler scheme) and the nabla-calculus with the backward
difference operator as derivative (and the implicit Euler scheme)
when the time scale $\mathbb{T}$ is $\mathbb{Z}$.
The duality between the two approaches,
on an arbitrary time scale $\mathbb{T}$,
has been recently exploited \citep{Caputo}.
Here we introduce the study of backward linear
control systems defined on an arbitrary time scale.
We claim that such systems are important in applications
because they rely on information about past values of states and/or outputs.
Indeed, as pointed out in \citep{MR2436490},
the nabla time scales analysis has important
implications for numerical analysts, who often use backward differences
instead of forward differences in their computations.
In this work controllability, observability, and realizability
conditions for nabla time-varying linear control systems
are proved using duality arguments and corresponding delta results.
Illustrative examples are given. We trust that the approach
here promoted can open further directions of future research.


\section*{Acknowledgement}

The authors were partially supported
by the \emph{Center for Research on Optimization and Control} (CEOC)
and by the \emph{Center of Research and Development in Mathematics and Applications}
(CIDMA) from the \emph{Portuguese Foundation for Science and Technology} (FCT),
cofinanced by the European Community fund FEDER/POCI 2010.

\medskip

The authors would like to express their gratitude to two anonymous
referees for relevant and stimulating remarks.



\end{document}